\documentclass[onefignum,onetabnum]{siamart220329}

\usepackage[utf8]{inputenc}
\usepackage{graphics}
\usepackage[mode=buildnew]{standalone}
\usepackage{xcolor}
\usepackage{amsmath}
\usepackage{tikz}
\usetikzlibrary{calc}
\usepackage{amssymb}
\usepackage{amsmath}

\newcommand{\poly}[1]{\Pi_{#1}}
\newcommand{\polyq}[2]{\Pi_{#1,#2}}

\usepackage{amsfonts}
\usepackage{graphicx}
\usepackage{epstopdf}
\usepackage{algorithmic}
\ifpdf
  \DeclareGraphicsExtensions{.eps,.pdf,.png,.jpg}
\else
  \DeclareGraphicsExtensions{.eps}
\fi

\newsiamremark{remark}{Remark}
\newsiamremark{hypothesis}{Hypothesis}
\crefname{hypothesis}{Hypothesis}{Hypotheses}
\newsiamthm{claim}{Claim}
\newsiamthm{conjecture}{Conjecture}

\headers{Corner cases of the tau method}{K. J. Burns, D. Fortunato, K. Julien, and G. M. Vasil}

\title{Corner cases of the tau method: symmetrically imposing boundary conditions on hypercubes}

\author{
    \mbox{Keaton J. Burns}\thanks{Corresponding author. Department of Mathematics, Massachusetts Institute of Technology, Cambridge, MA 02139 
      (\email{kjburns@mit.edu}).}
    \and \mbox{Daniel Fortunato}\thanks{Center for Computational Mathematics, Flatiron Institute, New York, NY 10010.}
    \and \mbox{Keith Julien}\thanks{The author is deceased. Former address: Department of Applied Mathematics, University of Colorado Boulder, Boulder, CO 80309.}
    \and \mbox{Geoffrey M. Vasil}\thanks{School of Mathematics, University of Edinburgh, Edinburgh, EH9 3FD, United Kingdom.}
}

\usepackage{amsopn}

\ifpdf
\hypersetup{
  pdftitle={Corner cases of the tau method: symmetrically imposing boundary conditions on hypercubes},
  pdfauthor={K. J. Burns, D. Fortunato, K. Julien, and G. M. Vasil}
}
\fi

\begin{document}

\maketitle

\begin{abstract}
Polynomial spectral methods produce fast, accurate, and flexible solvers for broad ranges of PDEs with one bounded dimension, where the incorporation of general boundary conditions is well understood.
However, automating extensions to domains with multiple bounded dimensions is challenging because of difficulties in imposing boundary conditions at shared edges and corners.
Past work has included various workarounds,  such as the anisotropic inclusion of partial boundary data at shared edges or approaches that only work for specific boundary conditions.
Here we present a general system for imposing boundary conditions for elliptic equations on hypercubes.
We take an approach based on the generalized tau method, which allows for a wide range of boundary conditions for many different spectral schemes.
The generalized tau method has the distinct advantage that the specified polynomial residual determines the exact algebraic solution; afterwards, any stable numerical scheme will find the same result. 
We can, therefore, provide one-to-one comparisons to traditional collocation and Galerkin methods within the tau framework. 
As an essential requirement, we add specific tau corrections to the boundary conditions, in addition to the bulk PDE, which produce a unique set of compatible boundary data at shared subsurfaces.
Our approach works with general boundary conditions that commute on intersecting subsurfaces, including Dirichlet, Neumann, Robin, and any combination of these on all boundaries.
The boundary tau corrections can be made hyperoctahedrally symmetric and easily incorporated into existing solvers.
We present the method explicitly for the Poisson equation in two and three dimensions and describe its extension to arbitrary elliptic equations (e.g. biharmonic) in any dimension.
\end{abstract}



\section{Introduction}

Global spectral methods are a powerful technique for solving parabolic and elliptic partial differential equations (PDEs) in simple domains \cite{trefethen2000spectral,Boyd.2001}.
For periodic domains, Fourier spectral methods produce diagonal discretizations of constant-coefficient linear differential operators and provide exponentially converging approximations to smooth solutions.
For bounded dimensions, orthogonal polynomials (such as Chebyshev polynomials) have similar convergence properties and can produce banded discretizations for linear differential operators with properly chosen test and trial bases \cite{Boyd.2001}.

Domains that are periodic in all but one dimension can use direct products of Fourier series and a single polynomial basis.
PDEs that are translationally invariant along the periodic dimensions will have discretizations that linearly separate across Fourier modes.
The resulting decoupled polynomial systems can be solved via collocation \cite{trefethen2000spectral} or coefficient-based methods, including spectral integration and Petrov-Galerkin schemes (such as the ultraspherical method \cite{Olver.2013}).
Boundary conditions are easily implemented at the two boundaries, and the production of fast solvers for generic equations can be robustly automated \cite{Burns.2020}.
These methods have a long history of application in the physical sciences, particularly to fundamental problems in fluid dynamics in ``channel'' geometries, such as the Orr-Somerfeld problem \cite{Orszag.19714ni}, transitions to turbulence \cite{Orszag.1980hv,Moin.1980,Trefethen.1993,GIBSON.2009}, and Rayleigh-Benard convection \cite{Tuckerman.1989,Julien.1996}.

For domains with more than one bounded dimension, global spectral methods become substantially more difficult to implement.
One fundamental issue is enforcing general boundary conditions in such a way that they are consistent with each other at shared edges and points.
In this work, we will consider problems on squares and cubes.
For the square, we will use edge-aligned coordinates $(x,y)$, and label the edges as north ($N$), east ($E$), south ($S$), and west ($W$), as illustrated in \cref{fig.domains}~(left).
For the cube, we will use face-aligned coordinates $(x,y,z)$, label the $x$ and $y$ faces as in the square, and label the $z$ faces as top ($T$) and bottom ($B$), as illustrated in \cref{fig.domains}~(center).

\begin{figure}
\centering
\includestandalone[height=1.75in]{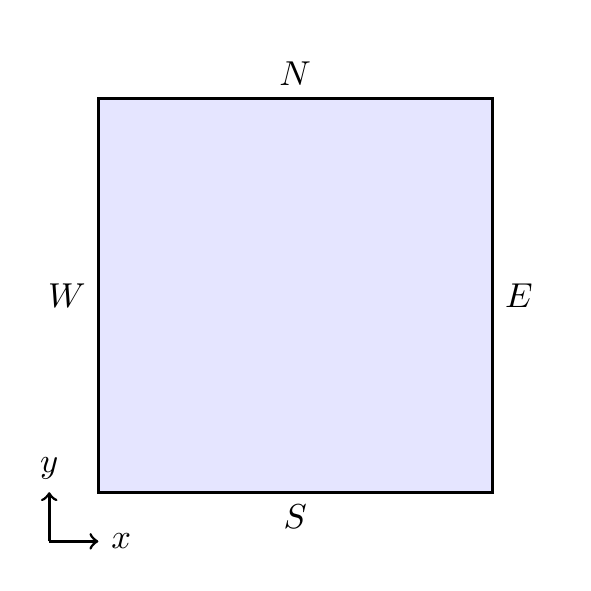}
\hspace{-2em}
\raisebox{0.3em}{\includestandalone[height=1.45in]{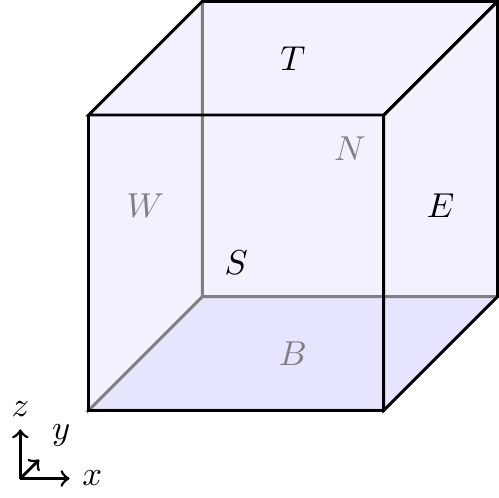}}
\hspace{-0.5em}
\includestandalone[height=1.75in]{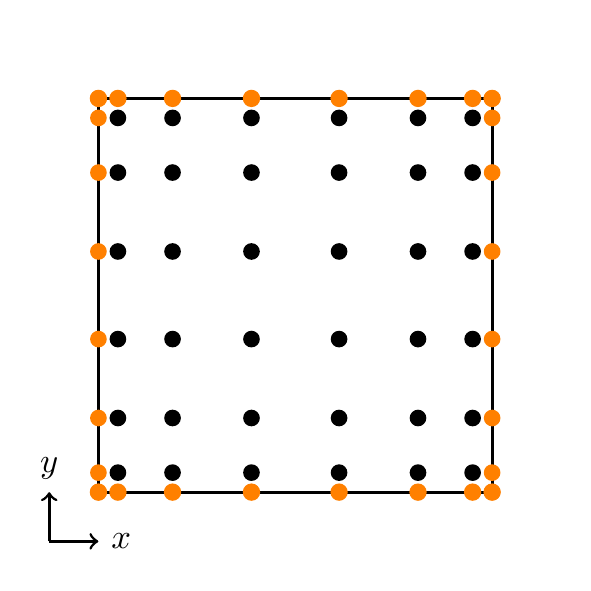}
\caption{
    Left: We consider problems in the square with coordinates $(x,y)$ and edges labeled north~($N$), east~($E$), south~($S$), and west~($W$).
    Center: We consider problems in the cube with coordinates $(x,y,z)$ and faces labeled north~($N$), east~($E$), south~($S$), west~($W$), top~($T$), and bottom~($B$).
    Right: Standard type-II (extrema) collocation discretizations on the square contain  $(N_x-2)(N_y-2)$ interior nodes (black) and $2 N_x + 2 N_y - 4$ boundary nodes (orange); all other spectral formulations of second-order elliptic problems require the same number of interior and boundary constraints.}
\label{fig.domains}
\end{figure}

For Dirichlet boundary conditions, continuity at corners and edges is readily achieved by enforcing the boundary conditions via collocation on type-II (extrema grids, which place nodes directly at the domain edges and corners (\cref{fig.domains}~(right)).
For a size $N_x \times N_y$ discretization of the square, there are $(N_x-2)(N_y-2)$ interior nodes and $2 N_x + 2 N_y - 4$ boundary nodes.
We will see that these numbers correspond to the independent interior and boundary degrees of freedom for all bivariate spectral methods for second-order elliptic PDEs on the square, but they are not always as easily partitioned as in the type-II collocation case.

Standard collocation schemes simply impose the interior PDE at the interior nodes and impose the Dirichlet conditions at the boundary nodes.
This is a form of row replacement, where the PDE residual conditions on the boundary are replaced by the boundary conditions.
No consideration of consistency conditions between adjoining boundary conditions is required, as a single Dirichlet value is enforced directly at each corner (e.g. \cite{Dang-Vu.1993}).
Of course, if the original problem is posed with Dirichlet conditions that are inconsistent at the corners (e.g.\ $u(x=E,y \to N) \neq u(x \to E, y=N)$), then some Dirichlet value must be chosen for the corner nodes (e.g.\ the average of these limits), but the linear system with this approach is nonsingular.

For other boundary conditions or methods, however, specifying boundary conditions in a nonsingular and well-conditioned way is more difficult.
For instance, what condition should be applied at the corner node if the adjacent boundaries have Neumann conditions, or a mix of different types of conditions?
For Galerkin methods\footnote{Here we use the term ``Galerkin'' to refer to any method using orthogonal polynomials of varying degrees as a basis, as opposed to collocation methods using Lagrange polynomials as a basis. We do not specifically refer to bases that \emph{a-priori} satisfy the boundary conditions, as the term is sometimes used.}, where no variables directly correspond to the solution values at the corners, how should boundary conditions be applied and how are the corners handled?

In these cases, boundary condition enforcement has tended to be ad-hoc and often fails to respect the dihedral symmetry of the square.
For instance, in collocation schemes with Neumann conditions, one may choose to enforce the conditions coming from just one of the incoming edges at each corner \cite{Gillman.2015}.
For Galerkin methods, the degenerate boundary conditions may all be imposed and solved via least squares \cite{Fortunato.2021}, subsets of boundary conditions may also be chosen for each edge \cite{Haidvogel.1979xt, Haldenwang.1984, Townsend.2015yvf}, or recombined bases satisfying the boundary conditions (sometimes called Galerkin bases) can be manually constructed \cite{Awan.1993,Julien.2009,Fortunato.2019}.
These approaches all provide valid and accurate solutions, but can be difficult to automate for high-order equations and mixed boundary conditions.

In this paper, we present a generalized tau scheme for incorporating general commuting boundary conditions for the Poisson equation in the square and cube.
The method adds ``tau terms'' to the boundary conditions as well as the PDE and uses additional constraints on shared edges and corners to ensure the boundary conditions are mutually compatible and properly close the tau-modified PDE.
Our scheme works with many different tau polynomials, including those corresponding to classical collocation and classical tau schemes.
The method extends to polyharmonic equations of arbitrary order in hypercubes of any dimension.
It also satisfies hyperoctahedral symmetries, which may aid the implementation of sparse spectral element schemes using these techniques within each element \cite{Fortunato.2021}.

We begin with a discussion of the generalized tau method in one dimension in \cref{sec.tau}.
We then describe the extension of this method to two dimensions for the Poisson equation in \cref{sec.2d} and develop dihedrally symmetric tau modifications for general commuting boundary conditions.
We further extend this approach to the Poisson equation in 3D in \cref{sec.3d}, where compatibility considerations are required on both the edges and corners of the cube.
In \cref{sec.gen}, we generalize this procedure to arbitrary-order elliptic operators in arbitrary-dimensional hypercubes.
Finally we conclude in \cref{sec.examples} with a variety of example problems including the Poisson and biharmonic equations in 2D and 3D.

\section{The generalized tau method}
\label{sec.tau}

We begin by reviewing the generalized tau method in 1D, which provides a mathematical framework for analyzing different polynomial spectral schemes by examining the residual they add to the PDE in order to accommodate boundary conditions.
Consider a 1D linear boundary value problem $\mathcal{L} u(x) = f(x)$ on $x \in [-1, 1]$, where $\mathcal{L}$ is a constant-coefficient order-$b$ elliptic operator.
The system is closed with $b$-many boundary conditions $\mathcal{B} u(x) = g$, where $\mathcal{B}$ is a column vector of $b$-many linear functionals.

Any weighted-residual spectral method with trial functions $\{\phi_i(x)\}$ and test functions $\{\psi_i(x)\}$ can be applied to this problem, resulting in a square discrete system after truncation at $N$ modes:
\begin{equation}
    \underbrace{\langle \psi_i | \mathcal{L} \phi_j \rangle}_{L_{ij}} \underbrace{\langle \phi_j | u \rangle}_{u^\phi_j} = \underbrace{\langle \psi_i | f \rangle}_{f^\psi_i}, \quad i, j = 0 .. N-1,
\end{equation}
where each inner product is that under which the left/bra functions are orthonormal.
Standard ``boundary-bordering'' or ``row-replacement'' schemes remove $b$-many of these constraints and replace them with the boundary conditions, resulting in a new square system of size $N$.
The tau method, by contrast, augments the original PDE with \emph{tau terms} containing $b$-many undetermined \emph{tau variables} $\{\tau_k \in \mathbb{R}\}$ multiplied by specified \emph{tau polynomials} $\{P_k(x)\}$ as:
\begin{equation}
    \mathcal{L} u(x) + \tau(x) = f(x), \quad \tau(x) = \sum_{k=1}^b \tau_k P_k(x).
\end{equation}
The boundary conditions are then applied alongside the weighted-residual discretization of this tau-modified system, producing a square system of size $N+b$ for the coefficients of $u$ as well as the unknown tau variables.
When the resulting system is solvable, the following important consequences hold:
\begin{itemize}
    \item The generalized tau method can be understood as modifying the original problem so that it has an exact finite-degree polynomial solution.
    \item The choice of the tau polynomials alone determines the exact polynomial solution of the modified equations.
    This same solution can then be found using any (polynomial) test and trial bases to discretize the perturbed system.
    \item The weighted-residual discretization of the modified system, with the addition of the boundary conditions, will be square and of size $N + b$.
\end{itemize}

To examine when these tau modifications produce a solvable system, we introduce the following notation.
We denote $\poly{N}$ as the vector space of real-valued polynomials with degree less than $N$ over a single real variable.
Note we do not include degree-$N$ polynomials in $\poly{N}$ merely for simplicity so that $\dim(\poly{N}) = N$ rather than $N+1$, which will simplify the following accounting.
We denote $\polyq{M}{N}$ as the quotient space $\poly{N} / \poly{M}$ with the equivalence relation $P \sim Q$ on $\polyq{M}{N}$ if $P,Q \in \poly{N}$ and $P-Q \in \poly{M}$.
For constant-coefficient linear differential equations in 1D, we can prove sufficient conditions for the tau-modified equations to be solvable.
We build towards that case by first considering a few simpler examples:

\begin{lemma}
\label{lemma.straight_derivatives}
Consider the differential equation $\partial^b u = f$ along with $b$-many linear boundary conditions $\mathcal{B} u = g$.
When $f \in \poly{N-b}$, the equation has a unique solution $u \in \poly{N}$ if the matrix $\mathcal{B} \begin{bmatrix} x^0 & ... & x^{b-1} \end{bmatrix}$ is full rank.
\end{lemma}

\begin{proof}
There is a unique order-$b$ antiderivative of $f$ with the form $x^b \tilde{f}$ with $\tilde{f} \in \poly{N-b}$.
Let $u = x^b \tilde{f} - \sum_{k=0}^{b-1} \alpha_k x^k$.
This function satisfies the differential equation by construction.
It satisfies the boundary conditions if $\mathcal{B} \begin{bmatrix} x^0 & ... & x^{b-1} \end{bmatrix} \begin{bmatrix} \alpha_0 & ... & \alpha_{b-1} \end{bmatrix}^T = \mathcal{B} x^b \tilde{f} - g$.
A unique satisfactory set of $\{\alpha_k\}$ coefficients can therefore be found for any $f$ and $g$ if the matrix $\mathcal{B} \begin{bmatrix} x^0 & ... & x^{b-1} \end{bmatrix}$ is full rank.
Uniqueness of $u$ follows since no polynomials with degree $\geq b$ are annihilated by the operator.
\end{proof}

\begin{lemma}
\label{lemma.diff_same_order}
Consider the differential equation $\partial^b u = f$ along with $b$-many linear boundary conditions $\mathcal{B} u = g$.
When $f \in \poly{N}$, the tau-modified equation $\partial^b u + \sum_{k=1}^b \tau_k P_k = f$ has a unique solution $u \in \poly{N}$ if the matrix $\mathcal{B} \begin{bmatrix} x^0 & ... & x^{b-1} \end{bmatrix}$ is full rank and if the tau polynomials $\{P_k\}$ span $\polyq{N-b}{N}$.
\end{lemma}

\begin{proof}
If the tau polynomials span $\polyq{N-b}{N}$, then there is a unique set of tau coefficients $\{\tau_k\}$ such that $f - \sum_{k=1}^b \tau_k P_k \equiv \hat{f} \in \poly{N-b}$.
The result then follows using Lemma \ref{lemma.straight_derivatives} for the equation $\partial^b u = \hat{f}$.
\end{proof}

\begin{lemma} 
Consider a constant-coefficient linear differential equation $L(\partial) u = f$, where $L$ is a degree-$b$ polynomial with $L(0) \neq 0$, along with $b$-many linear boundary conditions $\mathcal{B} u = g$.
When $f \in \poly{N}$, the tau-modified equation $L(\partial) u + \sum_{k=1}^b \tau_k P_k = f$ has a unique solution $u \in \poly{N}$ if $P_k \in \poly{N}$ and the matrix $\mathcal{B} \begin{bmatrix} p_1 & ... & p_b \end{bmatrix}$ is full rank, where $L(\partial) p_k = P_k$.
\end{lemma}

\begin{proof}
Since $L(0) \neq 0$, the differential operator is uniquely invertible over polynomials and maintains degree.
We can then define $u_f = L(\partial)^{-1} f \in \poly{N}$ and $p_k = L(\partial)^{-1} P_k \in \poly{N}$.
Let $u = u_f - \sum_{k=1}^b \tau_k p_k \in \poly{N}$.
This function satisfies the tau-modified differential equation by design.
It satisfies the boundary conditions if $\mathcal{B} \begin{bmatrix} p_1 & ... & p_b \end{bmatrix} \begin{bmatrix} \tau_1 & ... & \tau_b \end{bmatrix}^T = \mathcal{B} u_f - g$.
A unique satisfactory set of tau coefficients can therefore be found for any $f$ and $g$ if the matrix $\mathcal{B} \begin{bmatrix} p_1 & ... & p_b \end{bmatrix}$ is full rank.
Uniqueness of $u$ follows since no polynomials are annihilated by the operator.
\end{proof}

\begin{proposition}
Consider a constant-coefficient linear differential equation $L(\partial) u = f$, where $L$ is a degree-$b$ polynomial, along with $b$-many linear boundary conditions $\mathcal{B} u = g$.
Let $a$ be the degree of the greatest monomial factor $L$, so we can write $L(\partial) = \partial^a \tilde{L}(\partial)$ where $\deg \tilde{L} = b-a$ and $\tilde{L}(0) \neq 0$.
When $f \in \poly{N}$, the tau-modified equation $L(\partial) u + \sum_{k=1}^{a} \tau_k P_k + \sum_{k=a+1}^b \tau_k Q_k = f$ has a unique solution $u \in \poly{N}$ if the tau polynomials $P_k$ span $\polyq{N-a}{N}$, if $Q_k \in \poly{N-a}$, and if the matrix $\mathcal{B} \begin{bmatrix} x^0 & ... & x^{a-1} & q_{a+1} & ... & q_b \end{bmatrix}$ is full rank, where $L(\partial) q_k = Q_k$ and $x^a | \tilde{L}(\partial) q_k$.
\end{proposition}

\begin{proof}
If the tau polynomials $P_k$ span $\polyq{N-a}{N}$, then there is a unique set of tau coefficients $\{\tau_1, ...., \tau_a\}$ such that $f - \sum_{k=1}^a \tau_k P_k \equiv \hat{f} \in \poly{N-a}$.
There are unique order-$a$ antiderivatives of $\hat{f}$ and $Q_k$ with the form $x^a \tilde{f}$ and $x^a \tilde{Q}_k$ with $\tilde{f},\tilde{Q}_k \in \poly{N-a}$.
Since $\tilde{L}(0) \neq 0$, this operator is uniquely invertible over polynomials and maintaines degree.
We can then define $u_f = \tilde{L}(\partial)^{-1} x^a \tilde{f} \in \poly{N}$ and $q_k = \tilde{L}(\partial)^{-1} x^a \tilde{Q}_k \in \poly{N}$.
Let $u = u_f - \sum_{k=0}^{a-1} \alpha_k x^k - \sum_{k=a+1}^b \tau_k q_k \in \poly{N}$.
This function satisfies the tau-modified differential equation by design.
It satisfies the boundary conditions if $\mathcal{B} \begin{bmatrix} x^0 & ... & x^{a-1} & q_{a+1} & ... & q_b \end{bmatrix} \begin{bmatrix} \alpha_0 & ... & \alpha_{a-1} & \tau_{a+1} & ... & \tau_b \end{bmatrix}^T = \mathcal{B} u_f - g$.
A unique satisfactory set of coefficients can therefore be found for any $f$ and $g$ if the matrix $\mathcal{B} \begin{bmatrix} x^0 & ... & x^{a-1} & q_{a+1} & ... & q_b \end{bmatrix}$ is full rank.
Uniqueness of $u$ follows since no polynomials with degree $\geq a$ are annihilated by the operator.
\end{proof}

Proving sufficient conditions for solvability in higher dimensions or with nonconstant coefficients is more complex.
Here we will assume that a similar condition to Lemma \ref{lemma.diff_same_order} holds for the Poisson equation in arbitrary dimensions:

\begin{conjecture}
\label{conj.poisson_solvability}
Consider the Poisson equation $\Delta u = f$ in the hypercube of dimension $d$ with continuous Dirichlet boundary conditions.
We conjecture that the tau-modified equation $\Delta u + \sum_k \tau_k P_k = f$ has a unique solution $u \in \poly{N}^d$ if $f \in \poly{N}^d$ and the tau polynomials $P_k$ span
\begin{equation}
    \polyq{N-2}{N}^{(d)} = \poly{N}^d / \poly{N-2}^d.
\end{equation}
\end{conjecture}

\noindent This conjecture seems to be borne out by many numerical examples, and it forms the basis of our approach for other equations and boundary conditions.

All weighted residual spectral methods can be written as a discretization of a generalized tau form of the PDE with a given set of test and trail functions.
The \emph{classical tau method} picks the test and trial functions to be identical $\psi_i = \phi_i$, and chooses the tau polynomials as the $b$-many highest-degree elements of the test function set: $P_k = \phi_{N-k}$ for $k = 1 .. b$.
Only the choice of the tau polynomials, however, affects the resulting polynomial solution.
The choice of test and trial bases affect the form and conditioning of the resulting matrix, but in exact arithmetic do not impact the discrete solution.
This unification makes the generalized tau method particularly powerful for comparing different spectral methods since it distinguishes \textbf{what} approximate problem they solve from \textbf{how} they solve it.
Several other common spectral methods and their tau formulations include:

\begin{itemize}

    \item \emph{Standard collocation}, which takes the test and trial functions to be the Lagrange polynomials on the $N$-point type-II Chebyshev grid.
    For a typical second-order problem, the tau polynomials are taken to be the Lagrange polynomials supported at the endpoint nodes.
    Higher-order problems and integral conditions are handled in an ad-hoc manner, for instance including the Lagrange polynomials on additional near-edge nodes as tau polynomials.
    The resulting matrices are dense.
    
    \item \emph{Rectangular collocation}, which takes the trial basis to be the Lagrange polynomials on the $N$-point type-I Chebyshev grid and the test basis to be the Lagrange polynomials on the $(N-b)$-point type-I grid \cite{Driscoll.2015}.
    If the Chebyshev inner product is used, this is equivalent to the classical Chebyshev tau method with $P_k = T_{N-k}$ for $k=1 .. b$.
    If Gaussian quadrature on the $(N-b)$-point grid is used, then some aliasing errors are incurred.
    In both cases, the resulting matrices are dense.
    
    \item The \emph{ultraspherical method} takes the trial basis to be the ultraspherical polynomials $\{C^{(\alpha)}_i\}$, often specifically the Chebyshev polynomials $T_i = C_i^{(0)}$, and the test basis to be $\{C^{(\alpha+b)}_i\}$.
    This choice results in banded differential operators, enabling fast direct solvers for many systems of equations.
    Several choices of the tau polynomials, however, are still possible:
    
    \begin{itemize}
    
        \item The \emph{standard} or \emph{full-order} ultraspherical method takes the tau polynomials to coincide with the highest terms in the test basis: $P_k = C^{(\alpha+b)}_{N-k}$ for $k=1 .. b$.
        
        \item The \emph{first-order} ultraspherical method reduces the system to first-order form and adds tau terms in the $\{C^{(\alpha+1)}_i\}$ basis.
        The resulting system can be collapsed back into a high-order system before being projected against the test basis, resulting in a mix of tau terms from $\{\{C^{(\alpha+1)}_i\}, ..., \{C^{(\alpha+b)}_i\}\}$.
    
    \end{itemize}
    
\end{itemize}

We note that whenever the tau polynomials coincide with the test bases (including standard collocation and the standard ultraspherical method), the corresponding $b$-many residual constraints can simply be dropped and the solution can be determined from just the remaining residual constraints and the boundary conditions.
In terms of linear algebra, this corresponds to the case where the full system can be block triangularized.
Conversely, any row-replacement scheme can be directly interpreted as a tau method with tau polynomials corresponding to the dropped modes.
In these cases, the tau variables can be determined \emph{a-posteriori} from the corresponding residuals, if desired.
In the general case, however, the tau variables must be solved for simultaneously with the solution coefficients.

While the generalized tau method in one dimension is relatively well understood, and knowingly or unknowingly deployed in many contexts, its extension to multiple dimensions is not commonly referenced in the literature.
As in one dimension, multidimensional spectral schemes can typically be viewed as generalized tau schemes.
Here we explicitly take this approach to understand how general boundary conditions can be incorporated into such schemes, and how to deal with compatibility issues at shared edges and corners on the domain boundaries.

\section{2D Poisson}
\label{sec.2d}

We begin with the 2D Poisson equation on the square $(x,y) \in \Omega = [-1, 1]^2$:
\begin{equation}
    \Delta u(x,y) = f(x,y).
\end{equation}
We first consider the equation with Dirichlet boundary conditions on the edges:
\begin{equation}
    u(e) = f_e, \qquad e \in \delta \Omega = \{N,E,S,W\}.
\end{equation}
The prescribed Dirichlet values do not \emph{a-priori} need to be continuous at the corners $v \in V = \{NE, NW, SE, SW\}$.

We consider generalized tau discretizations of the equation using a direct-product trial basis of the form:
\begin{equation}
    u(x,y) = \sum_{i,j=0}^{N-1} u_{ij} \phi_i(x) \phi_j(y).
\end{equation}
We choose the same discretization size ($N_x = N_y = N$) in both dimensions purely for notational simplicity.
All results easily generalize to anisotropic truncations.
We will see that we will need to add tau terms to both the PDE and the boundary conditions to produce a nonsingular system for the $N^2$ total degrees of freedom in $u$.

\subsection{Interior tau modifications}

The two-dimensional extension of the generalized tau method is to add tau terms to the PDE covering the boundary (for collocation) or highest degree (for a Galerkin method) polynomials in both dimensions.
Specifically, following Conjecture \ref{conj.poisson_solvability} the tau polynomials should span the space
\begin{equation}
    \polyq{N-2}{N}^{(2)} \equiv (\Pi_{N-2,N} \otimes \Pi_{N}) \cup (\Pi_{N} \otimes \Pi_{N-2,N})
\end{equation}
The two component spaces have a non-trivial intersection ($\Pi_{N-2,N} \otimes \Pi_{N-2,N}$), and so we can equivalently write
\begin{equation}
    \polyq{N-2}{N}^{(2)} = (\Pi_{N-2,N} \otimes \Pi_{N-2} ) \oplus (\Pi_{N-2} \otimes \Pi_{N-2,N} ) \oplus (\Pi_{N-2,N} \otimes \Pi_{N-2,N}).
\end{equation}
A typical tau term spanning this space takes the form
\begin{equation}
    \tau(x,y) = \sum_{k=1}^2 \left(\tau_{k}^{(x)}(x) P_k(y) + \tau_{k}^{(y)}(y) P_k(x)\right) + \sum_{k,l=1}^{2} \tau_{kl} P_k(x) P_l(y),
\end{equation}
where $\tau_{k}^{(x)}, \tau_{k}^{(y)} \in \poly{N-2}$, $\tau_{kl} \in \mathbb{R}$, and $\mathrm{span}\{P_k\} = \polyq{N-2}{N}$.
If the tau polynomials are chosen as the last modes of the test basis ($P_k = \psi_{N-k}$), this is equivalent to dropping the last two rows and columns of the equations in spectral space, as illustrated in \cref{fig.constraints_2d}.

\begin{figure}
\centering
\includestandalone[width=0.99\linewidth]{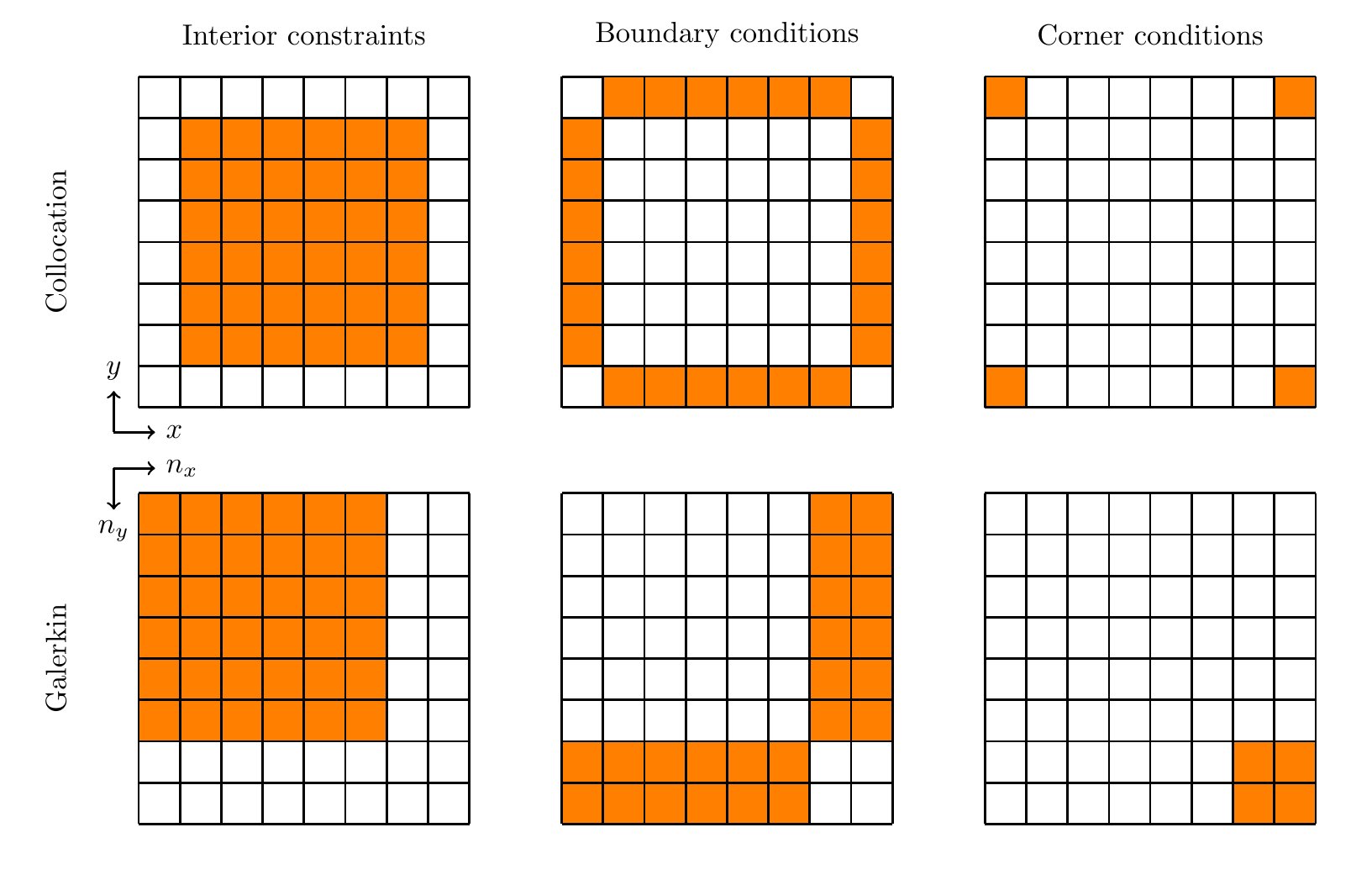}
\caption{
    A pictorial representation of the constraints in collocation (top) and Galerkin (bottom) tau schemes for the Poisson equation in 2D.
    From left to right:
    The interior equations are enforced on the $N-2$ interior nodes / low modes providing $(N-2)^2$ constraints.
    The boundary conditions on each edge are enforced on the $N-2$ interior nodes / low modes providing $4(N-2)$ constraints.
    Conditions on each corner provide the last $4$ constraints.}
\label{fig.constraints_2d}
\end{figure}

These tau terms consist of $4 (N-2) + 4 = 4 N - 4$ independent degrees of freedom, the same as the number of boundary nodes in the type-II collocation scheme.
As in that scheme, $(N-2)^2$ interior constraints on $u$ remain in the tau-modified PDE.
A consistent set of boundary conditions must determine the remaining degrees of freedom in $u$.

\subsection{Boundary tau modifications for Dirichlet data}

With $4N - 4$ interior tau terms, the same number of constraints from boundary conditions must be supplied.
However, the naive application of all boundary conditions to full order consists of $4N$ independent constraints.
Since the underlying spectral representation is continuous at the corners, the specified boundary values must generally also be adjusted to be continuous there.
This requires removing four of the boundary constraints that are independently supported at the corners, or, equivalently, adding additional tau terms to the boundary conditions to absorb any potential discontinuities in the specified boundary data.
Specifically, the boundary conditions are modified to take the form
\begin{equation}
    u(e) + \tau_e = f_e, \qquad e \in \delta \Omega,
\end{equation}
and we now aim to identify forms for each $\tau_e \in \poly{N}$ to render the system solvable.
We first examine how the standard collocation scheme corresponds to such an approach, and then take a similar route to form Galerkin tau schemes for Dirichlet data.

\subsubsection{Collocation modifications as a tau scheme}

In a type-II collocation scheme, there are only $4N - 4$ boundary points to begin with; these schemes force the imposition of continuous boundary data because adjacent edges share the corner collocation nodes.
If initially discontinuous boundary conditions are specified for the PDE, for instance if $f_N = 1$ and $f_W = 0$, then some unique Dirichlet value at the corner $NW$ must be chosen, perhaps taking either incoming value or their average.
Any particular choice will correspond to a choice of tau modifications added to the full boundary conditions.
For example, if the corner condition is set to $u_{NW} = f_{N}(W)$, then the prescribed value of $f_{W}(N)$ is ignored.
This corresponds to adding a tau term to the $W$ boundary condition with the tau polynomial being the Lagrange polynomial supported at the point $NW$, and the corresponding tau variable will be the difference between the imposed and ignored corner values.

Following this intuition, we expect that a general tau scheme should add four tau terms to the boundary conditions.
These terms should adjust the prescribed boundary data to absorb any discontinuities in the corners.

\subsubsection{Singularity of the simplest tau modifications}

On first inspection, it may seem that a reasonable extension of this approach to Galerkin schemes would be to add a single tau polynomial, for instance $\tau_e = \tau_1^{(e)} Q_1$ with $Q_1 = \psi_{N-1}$, to the boundary condition on each edge, as illustrated in \cref{fig.square_tau}~(left).
However, the resulting system is still singular if $Q_1(-1) = \pm Q_1(1)$, as is the case for orthogonal polynomials with an even weight function, because a continuous linear combination of these tau terms can be freely added to the solution.
That is,  the tau variables here are not uniquely determined by requiring continuity at the four corners; 
%$(\tau_T, \tau_R, \tau_B, \tau_L) \to (\tau_T + A, \tau_R + A, \tau_B \pm A, \tau_L \pm A)$ 
there are free adjustments to the tau variables that maintain continuity at all corners (see \cref{fig.square_tau}~(right)).
This issue arises whenever a continuous combination of the chosen tau polynomials can be constructed.

\begin{figure}
\centering
\includestandalone[width=0.31\linewidth]{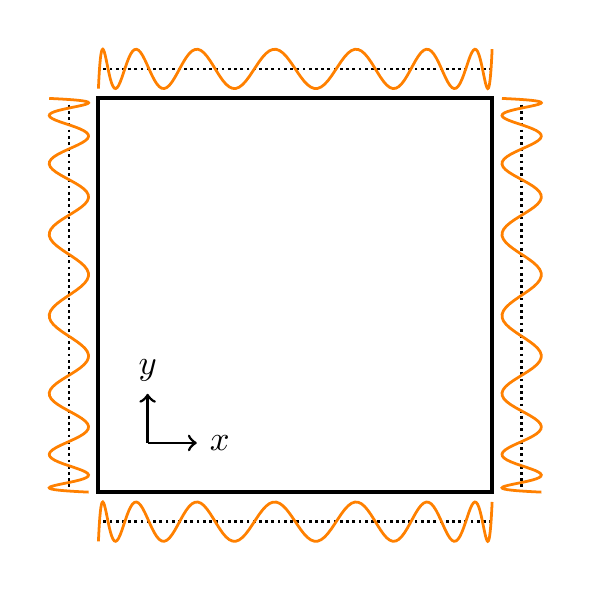}
\includestandalone[width=0.31\linewidth]{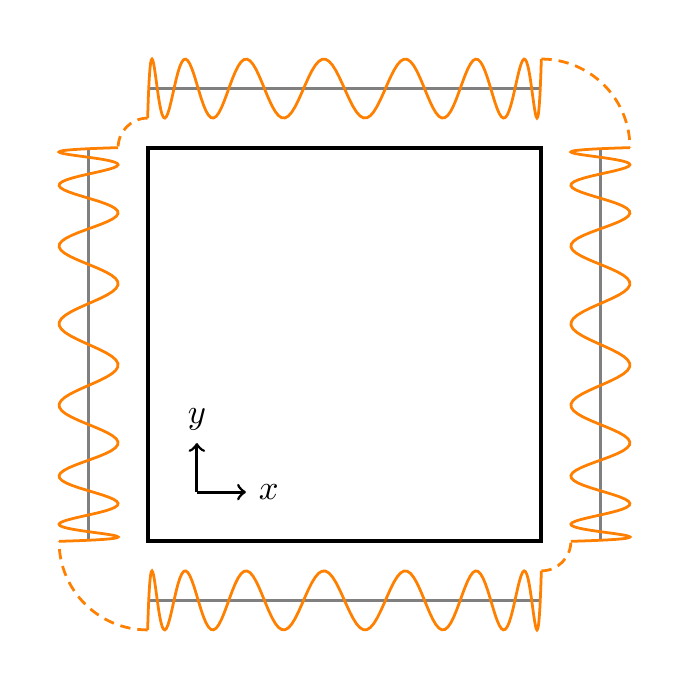}
\caption{Left: A naive approach to removing four constraints from the boundary conditions by placing a single tau polynomial (or dropping the highest mode) on each edge.
    Right: This approach is singular, because there is an unconstrained combination of these modes which is continuous at all four corners.}
\label{fig.square_tau}
\end{figure}

\subsubsection{Nonsingular boundary tau modifications}

Since the simplest Galerkin tau adjustment suffers from the interaction between tau terms on the various edges, an alternative approach is to construct tau terms that can individually absorb discontinuities in the boundary data at each corner without affecting the other corners.
The tau polynomial for each corner should satisfy the following conditions:
\begin{itemize}
    \item It should have a jump in value between the two incoming edges on the corresponding corner.
    \item It should be zero on all edges at all other corners.
\end{itemize}
If these conditions are met, no continuous combinations of the tau polynomials exist.
Each tau variable will be directly determined by the jump in the prescribed boundary data at the corresponding corner.
The boundary data will effectively be adjusted to ensure continuity by subtracting out the resulting tau terms, leaving $4N - 4$ remaining constraints on $u$ from the boundary conditions.

\begin{figure}
\centering
\includestandalone[width=0.31\linewidth]{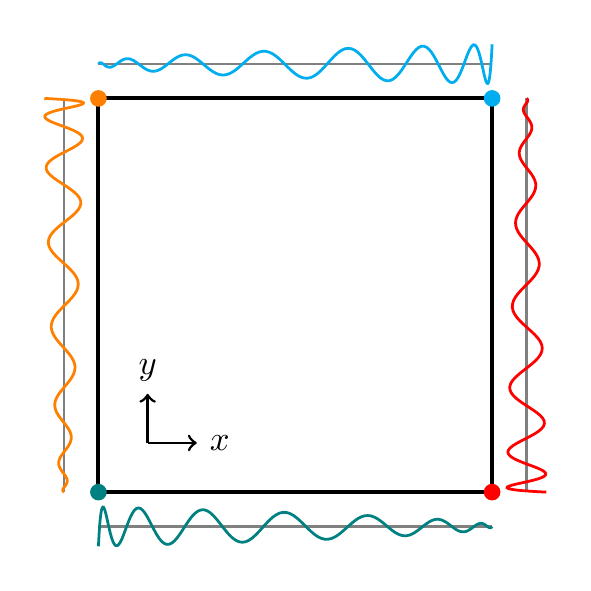}
\hfill
\includestandalone[width=0.31\linewidth]{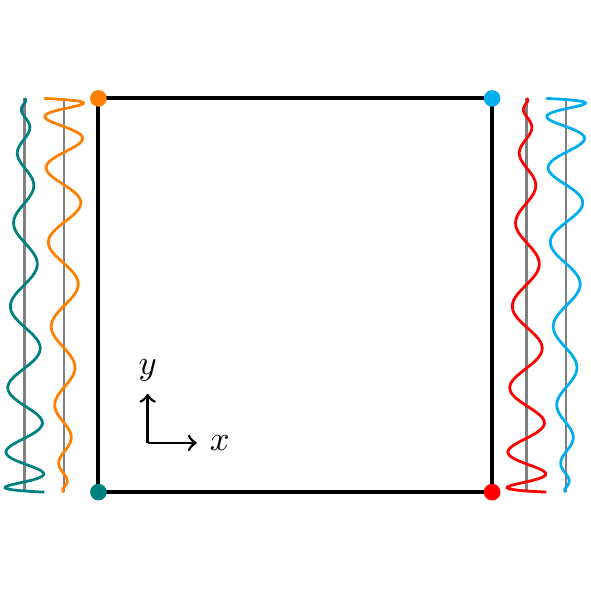}
\hfill
\includestandalone[width=0.31\linewidth]{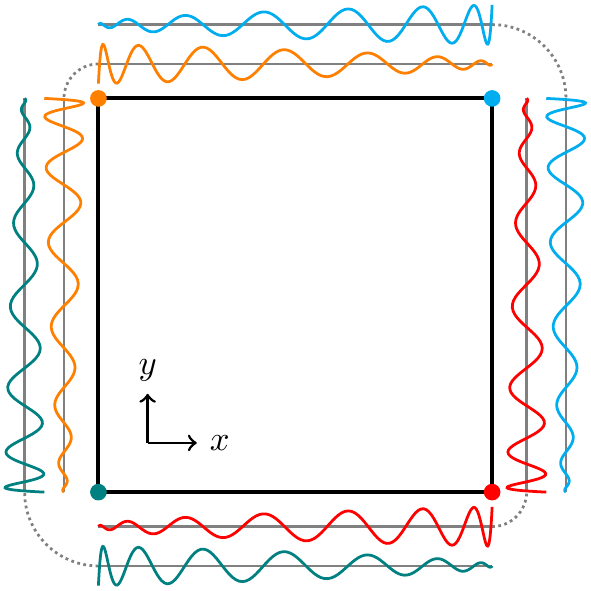}
\caption{Nonsingular tau modifications where each tau polynomial absorbs the discontinuity at a single corner.
    Left: The term on each edge absorbs discontinuities on the adjacent corner in the clockwise direction.
    Center: The north and south conditions are imposed exactly while the tau terms adjust the east and west conditions to match.
    Right: Dihedrally symmetric modifications corresponding to ``sawtooth'' functions at each corner.}
\label{fig.nonsingular_taus}
\end{figure}

Several potential choices for such tau polynomials are shown in \cref{fig.nonsingular_taus}.
Here we have chosen tau polynomials which all correspond to linear combinations of $Q_1 = T_{N-1}$ and $Q_2 = T_{N-2}$ on each edge, but other orthogonal polynomials or the endpoint-supported Lagrange polynomials could easily be substituted.
The left choice contains tau polynomials that are nonzero only on the clockwise corner relative to each edge.
This choice corresponds to picking the Dirichlet values at the corners to be those from the incoming edge data on the clockwise side and adjusting the other incoming edge data to match (as done in \cite{Gillman.2015}).
The center choice contains tau polynomials only on the east and west edges.
This choice corresponds to picking the Dirichlet values at the corners from the north and south edge data, and adjusting the east and west edge data to match (as done in \cite{Julien.2009}).
These choices have direct analogies in collocation schemes and have both found use in the literature, but they fundamentally violate the dihedral symmetry of the square.

A dihedrally symmetric choice is to place odd ``sawtooth''-like modifications on the two edges adjacent to each corner.
For instance, the tau polynomial for $NW$ could include the combination of $Q_1$ and $Q_2$ on $W$ that is 0 on $SW$ and 1 on $NW$, and the combination of $Q_1$ and $Q_2$ on $N$ that is 0 on $NE$ and -1 on $NW$.
% For instance, for definite-parity orthogonal polynomials normalized so that $Q_n(\pm 1) = (\pm 1)^n$, the correction on the north-west corner would include $\tau_{NW} (Q_{N-1}(y) + Q_{N-2}(y))$ on the west edge and $\tau_{NW} (-1)^N (Q_{N-1}(x) - Q_{N-2}(x))$ on the north edge.
% In total, the boundary tau terms would be:
% %
% \begin{equation}
% \begin{aligned}
%     \tau_W &= \tau_{SW} (-1)^N (Q_{N-1} - Q_{N-2}) + \tau_{NW} (Q_{N-1} + Q_{N-2}) \\
%     &= (\tau_{NW} + (-1)^N \tau_{SW}) Q_{N-1} + (\tau_{NW} - (-1)^N \tau_{SW})  Q_{N-2}
% \end{aligned}
% \end{equation}
This choice is shown in \cref{fig.nonsingular_taus}~(right), and corresponds to picking the Dirichlet values at the corners to be equal to the mean of the values from each incoming edge, and adjusting both incoming edges to match this mean.

\subsubsection{Boundary tau modifications via additional constraints}

We note that on each edge, the tau modifications from the adjacent corners consist of two different linear combinations of $Q_1$ and $Q_2$, which together span $\polyq{N-2}{2}$.
A simple approach to constructing this same set of modifications is therefore to add 8 total tau terms to the boundary conditions, one for each $Q$ polynomial on each edge, and include 4 additional constraints to restrict these additions to the desired combinations.
The tau terms on each edge would then be:
\begin{equation}
    \tau_e = \sum_{k=1}^2 \tau_k^{(e)} Q_k \in \polyq{N-2}{N}.
\end{equation}
The problem is closed by specifying 4 additional constraints, namely the corner conditions satisfied by the dihedral tau modifications, which are that the tau terms sum to zero at each corner $v$:
\begin{equation}\label{eq.dirichlet_corner_tau}
    \sum_{e \sim v} \tau_{e}(v) = 0 \qquad \forall v.
\end{equation}
Adding the tau-modified boundary equations for each $e \sim v$ and subtracting the tau constraint above, we see that we can equivalently impose at each corner $v$:
\begin{equation}\label{eq.dirichlet_corner_eqn}
    u(v) = \frac{1}{2} \sum_{e \sim v} f_e \qquad \forall v.
\end{equation}

To summarize, a dihedrally symmetric tau correction to the boundary conditions can be manually constructed using sawtooth-like tau functions.
However, an exactly identical scheme can be implemented by adding two tau polynomials to each edge and imposing additional constraints at each corner.
These constraints can be that the tau terms sum to zero at each corner, or equivalently that the corner values of the solution are the mean of the prescribed values from each incoming edge.
The boundary conditions then provide $4N - 4$ independent constraints on the interior solution, yielding a nonsingular system.

\subsection{Corner conditions for general commuting boundary conditions}

Following the intuition from the dihedral tau modifications for Dirichlet data, we now consider the 2D Poisson equation augmented with general commuting linear boundary conditions at each edge $e$:
\begin{equation}
    \beta_e u = g_e, \qquad e \in \delta \Omega.
\end{equation}
By commuting boundary conditions, we mean that the boundary operators on all adjacent edges $e$ and $e'$ must commute: $\beta_e \beta_{e'} = \beta_e \beta_{e'}$.
For instance, the boundary operators may independently be for each edge:
\begin{itemize}
    \item Dirichlet: $\beta_e = I_e$, where $I_e$ is the interpolation operator to edge $e$,
    \item Neumann: $\beta_e = I_e \partial_{n_e}$, where $\partial_{n_e}$ is the derivative normal to edge $e$,
    \item Robin: $\beta_e = I_e (a + b \partial_{n_e})$).
\end{itemize}

As before, $4N - 4$ interior tau terms are added to the PDE and the same number of constraints must be imposed from the boundary conditions.
We again impose the boundary conditions with two tau terms each as:
\begin{equation}
    \beta_e u + \tau_e = g_e, \qquad \forall e \in \delta \Omega.
\end{equation}
where $\tau_e = \sum_{k=1}^{2} \tau_{k}^{(e)} Q_{k}$ and $\mathrm{span}\{Q_k\} = \polyq{N-2}{N}$.
We have added $4N$ constraints but also $8$ degrees of freedom, so the system remains under-determined by 4 constraints.
Instead of simply imposing that the tau terms sum to zero in the corners as in the Dirichlet case, we must now use a more general approach for developing corner conditions to make the system uniquely solvable.

We can understand the possible sets of corner conditions by first considering the commutation of the boundary operators on adjacent edges.
Applying the complementary boundary operators to the boundary conditions on adjacent edges $e$ and $e'$ and subtracting gives:
\begin{eqnarray}
    \underbrace{(\beta_e \beta_{e'} - \beta_{e'} \beta_e) u}_{0} + \beta_e \tau_{e'} - \beta_{e'} \tau_e = \beta_e g_{e'} - \beta_{e'} g_e
\end{eqnarray}
The first term cancels because the boundary operators commute.
The remainders of these four equations describe how jumps in the tau terms on adjacent edges absorb incompatibilities in the prescribed boundary data.
Note that these are not additional constraints -- they come identically from the tau-modified boundary conditions that are already part of the so-far underdetermined system.

We need to pick four additional corner conditions which, along with these four corner jump conditions, form a solvable system for the 8 degrees of freedom in the boundary tau variables.
%
% For the system to be valid for any prescribed boundary data, the collection of these conditions for all four corners must form a nonsingular system for the tau variables:
% %
% \begin{equation}
%     \begin{bmatrix}
%         \beta_E & -\beta_N & 0 & 0 \\
%         0 & \beta_S & -\beta_E & 0 \\
%         0 & 0 & \beta_W & -\beta_S \\
%         -\beta_W & 0 & 0 & \beta_N \\
%     \end{bmatrix}
%     \begin{bmatrix}
%         \tau_N \\
%         \tau_E \\
%         \tau_S \\
%         \tau_W
%     \end{bmatrix}
%     = 
%     \begin{bmatrix}
%         \beta_E g_N - \beta_N g_E \\
%         \beta_S g_E - \beta_E g_S \\
%         \beta_W g_S - \beta_S g_W \\
%         \beta_N g_W - \beta_W g_N \\
%     \end{bmatrix}
% \end{equation}
% %
% .
%
In general, we will take the corner conditions to be additional linear constraints between the tau terms at each corner like:
\begin{equation}
    \alpha_{ee'} \beta_e \tau_{e'} + \alpha_{e'e} \beta_{e'} \tau_e = 0.
\end{equation}
The various choices for tau modifications analogous to those in \cref{fig.square_tau} can all be written in this form:
\begin{itemize}

    \item The clockwise tau scheme corresponds to taking $\alpha_{NW} = \alpha_{EN} = \alpha_{SE} = \alpha_{WS} = 0$ and $\alpha_{WN} = \alpha_{NE} = \alpha_{ES} = \alpha_{SW} = 1$.
    This results in the constraints $\beta_W \tau_N = \beta_N \tau_E = \beta_E \tau_S = \beta_S \tau_W = 0$.
    In a row-replacement scheme, these four constraints can be implemented without referencing the tau variables by subtracting them from the corresponding cross-evaluations of the boundary conditions, i.e. directly imposing:
    \begin{eqnarray}
        \beta_W \beta_N u & = \beta_W g_N \\
        \beta_N \beta_E u & = \beta_N g_E \\
        \beta_E \beta_S u & = \beta_E g_S \\
        \beta_S \beta_W u & = \beta_S g_W
    \end{eqnarray}
    
    \item The east-west tau scheme corresponds to taking $\alpha_{NE} = \alpha_{SE} = \alpha_{NW} = \alpha_{SW} = 0$ and $\alpha_{EN} = \alpha_{ES} = \alpha_{WN} = \alpha_{WS} = 1$.
    This results in the constraints $\tau_N = \tau_S = 0$ and $\tau_E$ and $\tau_W$ are fully determined by continuity in the corners.
    In a row-replacement scheme, this can be implemented without referencing the tau variables by imposing the top and bottom boundary conditions at the corners, i.e.:
    \begin{align}
        \beta_W \beta_N u & = \beta_W g_N \\
        \beta_E \beta_N u & = \beta_E g_N \\
        \beta_W \beta_S u & = \beta_W g_S \\
        \beta_E \beta_S u & = \beta_E g_S
    \end{align}
    
    \item The dihedral scheme corresponds to taking all $\alpha_{ee'} = 1$.
    In a row-replacement scheme, these conditions can be implemented without referencing the tau variables by adding the complementary boundary conditions at the corners and subtracting the tau constraints, i.e.:
    \begin{equation}\label{eq.isotropic_corners}
        2 \beta_{e} \beta_{e'} u = \beta_e g_{e'} + \beta_{e'} g_e
    \end{equation}
    Note that this expression reduces to \eqref{eq.dirichlet_corner_eqn} for Dirichlet conditions.
    
\end{itemize}

\subsection{The dihedrally symmetric ultraspherical tau method}

In summary, the full set of equations for a dihedrally symmetric tau scheme for the 2D Poisson equation with arbitrary commuting boundary conditions is:
\begin{equation}
\begin{gathered}
    \Delta u(x,y) + \tau(x,y) = f(x,y) \\
    \beta_e u + \tau_e = g_e \qquad \forall e\\
    2 \beta_{e} \beta_{e'} u = \beta_e g_{e'} + \beta_{e'} g_e \qquad \forall v
\end{gathered}
\end{equation}
with $\tau \in \polyq{N-2}{2}^{(2)}$ and $\tau_e \in \polyq{N-2}{2}$.

A dihedral version of the standard ultraspherical spectral method for the square can be constructed by taking:
\begin{itemize}
    \item Chebyshev polynomials as the trial basis: $\phi_i = T_i$,
    \item index-2 ultraspherical polynomials as the interior tau polynomials: $P_k = C^{(2)}_{N-k}$,
    \item Chebyshev polynomials as the boundary tau polynomials: $Q_k = T_{N-k}$.
\end{itemize}
Choosing the test bases the same as the tau polynomials produces a sparse interior operator and allows for $u$ to be solved independently of the tau variables in row-replacement fashion.

\begin{figure}
\centering
\includegraphics[width=0.4\linewidth]{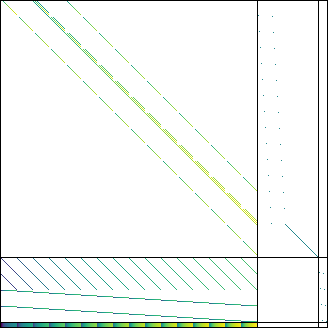}
\hspace{1cm}
\includegraphics[width=0.4\linewidth]{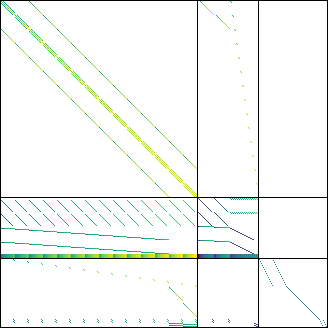}
\caption{Matrices for the 2D Poisson equation with Robin boundary conditions and $N=16$ using the dihedral ultraspherical tau discretization.
Left: the ``natural'' ordering with row blocks corresponding to the interior equation, the boundary conditions, and the corner conditions, and with column blocks corresponding to the solution variables, interior tau variables, and boundary tau variables. 
Right: a permuted system separating the taus and enabling efficient linear solves via the Schur complement.
The row blocks correspond to the low and middle modes of the interior equation, the low and middle modes of the boundary conditions and the corner conditions, and the high modes containing all tau terms.
The column blocks correspond to the middle and high solution modes, the low solution modes, and the interior and boundary tau terms.}
\label{fig.matrix_partitions}
\end{figure}

The full linear system for this discretization of the 2D Poisson equation with Robin boundary conditions and $N=16$ is illustrated in \cref{fig.matrix_partitions}.
The left panel shows the matrix with ``natural'' row and column orderings.
The row blocks are, in order: 
\begin{itemize}
    \item the interior equation: $N^2$ constraints, 
    \item the boundary conditions: $4N$ constraints, 
    \item the corner conditions in the manner of \cref{eq.isotropic_corners}: $4$ constraints.
\end{itemize}
The column blocks are, in order:
\begin{itemize}
    \item the solution $u$: $N^2$ variables,
    \item the interior tau terms: $4N - 4$ variables,
    \item and the boundary tau terms: $8$ variables.
\end{itemize}
Together these form a nonsingular system of size $(N+2)^2$ that can be simultaneously solved for the interior solution and tau variables given any RHS data.

Since the tau polynomials are from the family of test functions, the tau terms can be block-separated from the rest of the system.
The remaining terms can also be partitioned to separate the low-degree modes of the solution from the rest, leading to an efficient Schur-complement based solver for the interior solution.
This permutation is shown in the right panel of \cref{fig.matrix_partitions}.
The row blocks are now in the order:
\begin{itemize}
    \item the low and middle modes ($\max(i,j)<N-2$) of the interior equation: $(N-2)^2 = N^2 - 4N + 4$ constraints,
    \item the low and middle modes of the boundary conditions ($i<N-2$) and the corner conditions: $4(N-2) + 4 = 4N - 4$ constraints,
    \item the high modes of the interior equation ($\max(i,j) \ge N-2$) and boundary conditions ($i \ge N-2$), containing all the tau terms: $4N + 4$ constraints.
\end{itemize}
The column blocks are, in order:
\begin{itemize}
    \item the middle and high solution modes ($\min(i,j)\ge2$): $(N-2)^2 = N^2 - 4N + 4$ variables,
    \item the low solution modes ($\min(i,j)<2$): $4N - 4$ variables,
    \item the interior and boundary tau terms: $4N + 4$ variables.
\end{itemize}
Since the tau terms are block separated from the rest of the system, the solution $u$ can be found by just solving the principal submatrix consisting of the first two block rows and columns, which form a nonsingular square system of size $N^2$.
If desired, the tau variables can then be solved afterwards from the last block row.
This principal submatrix can be solved efficiently by taking the Schur complement of the first block -- the portion of the Laplacian taking the middle and high modes of the input to the low and middle modes of the output.
This block is a square, nonsingular, and well-conditioned matrix of size $(N-2)^2$.
This matrix has tridiagonal-kronecker-tridiagonal form, and can be solved in quasi-optimal time using an ADI scheme \cite{Fortunato.2019}.

\section{3D Poisson}
\label{sec.3d}

The strategy of adding tau terms to the boundary conditions and imposing constraints at their intersections extends nicely to higher dimensions.
We now consider the 3D Poisson equation on the cube $(x,y,z) \in \Omega = [-1,1]^3$:
\begin{equation}
    \Delta u(x,y,z) = f(x,y,z).
\end{equation}
We again consider generalized tau discretizations of the equation using direct-product trial bases of the form:
\begin{equation}
    u(x,y,z) = \sum_{i,j,k=0}^{N-1} u_{ijk} \phi_i(x) \phi_j(y) \phi_k(z),
\end{equation}
which contains $N^3$ degrees of freedom.

\subsection{Interior tau corrections}

Similar to the 2D case, we add tau terms to the PDE which span $\polyq{N-2}{N}^{(3)}$, namely:
\begin{equation}
\begin{aligned}
    \tau(x,y,z) = &\sum_{i=1}^2 \left[\tau_{yz}^{(i)}(y,z) P_i(x) + \tau_{zx}^{(i)}(z,x) P_i(y) + \tau_{xy}^{(i)}(x,y) P_i(z)\right] \\
    + &\sum_{i,j=1}^2 \left[\tau_{z}^{(ij)}(z) P_i(x) P_j(y) + \tau_{x}^{(ij)}(x) P_i(y) P_j(z) + \tau_{y}^{(ij)}(y) P_i(x) P_j(z)\right] \\
    + &\sum_{i,j,k=1}^2 \tau^{(ijk)} P_i(x) P_j(y) P_k(z),
\end{aligned}
\end{equation}
where $\tau_{yz}^{(i)}, \tau_{zx}^{(i)}, \tau_{xy}^{(i)} \in \Pi_{N-2}^2$, $\tau_{x}^{(ij)}, \tau_{y}^{(ij)}, \tau_{z}^{(ij)} \in \Pi_{N-2}$, $\tau^{(ijk)} \in \mathbb{R}$, and $\mathrm{span}\{P_i\} = \Pi_{N-2,N}$.
If the tau polynomials are chosen as the last modes of the test basis ($P_i = \psi_{N-i}$), this is equivalent to dropping the last two planes of the equations in each dimension in spectral space, as illustrated in \cref{fig.constraints_3d}.

\begin{figure}
\centering
\includestandalone[width=0.99\linewidth]{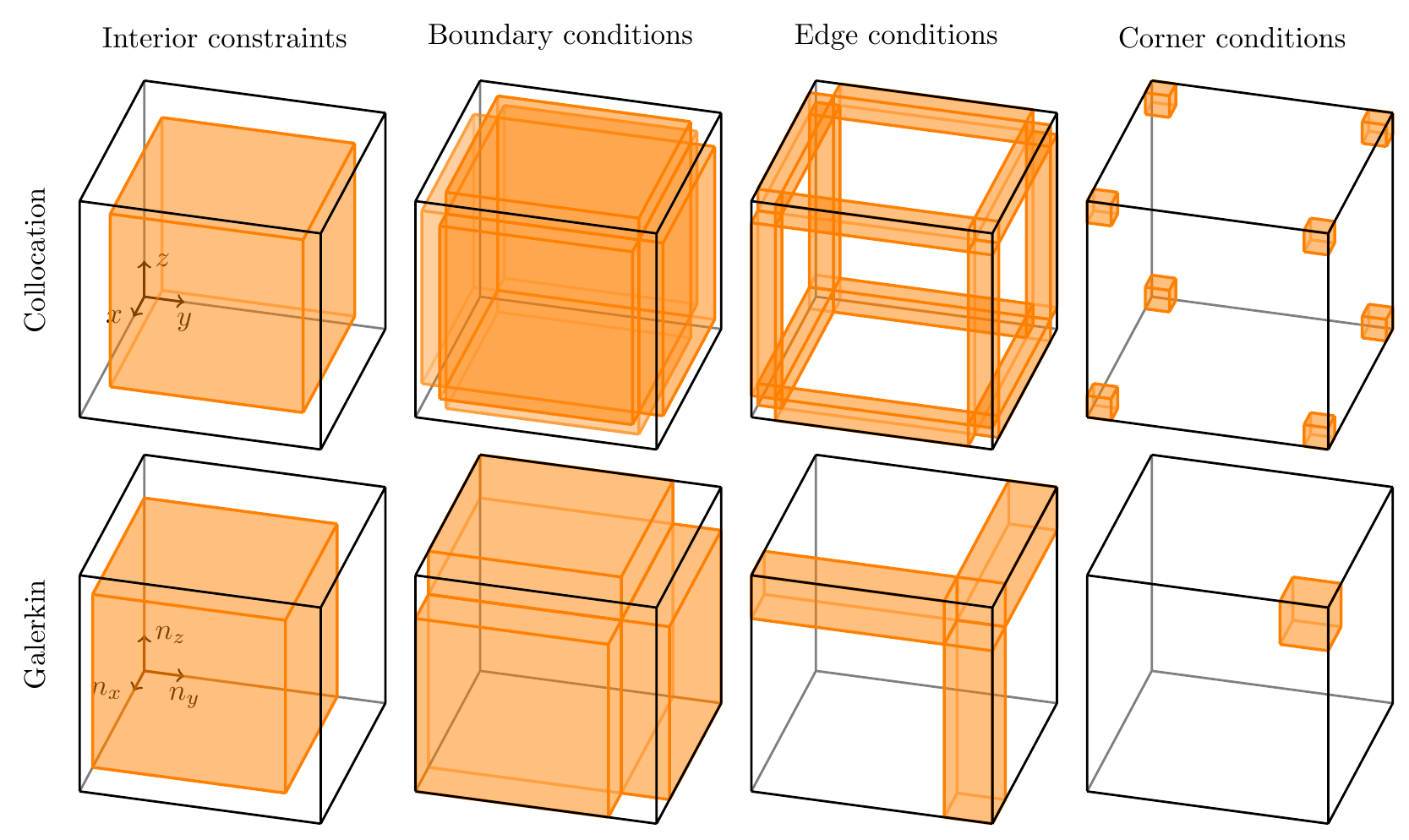}
\caption{A pictorial representation of the constraints in collocation (top) and Galerkin (bottom) tau schemes for the Poisson equation in 3D.
From left to right:
The interior equations are enforced on the $N-2$ interior nodes / low modes providing $(N-2)^3$ constraints.
The boundary conditions on each face are enforced on the $N-2$ interior nodes / low modes providing $6(N-2)^2$ constraints.
Conditions on each edge are enforced on the $N-2$ interior nodes / low modes providing $12(N-2)$ constraints.
Finally, conditions on each corner provide the last $8$ constraints.}
\label{fig.constraints_3d}
\end{figure}

These tau terms consist of $6(N-2)^2 + 12(N-2) + 8$ independent degrees of freedom, the same as the number of boundary nodes in the type-II collocation scheme.
As in that scheme, $(N-2)^3$ interior constraints remain in the tau-modified PDE.
A consistent set of boundary conditions must determine the remaining degrees of freedom in $u$.

\subsection{Corrections for general commuting boundary conditions}

We again consider boundary conditions for each face which commute with the boundary operators on all adjacent faces:
\begin{equation}
    \beta_f u = g_f \qquad \forall f \in \delta \Omega,
\end{equation}
\begin{equation}
    \beta_f \beta_{f'} = \beta_{f'} \beta_f \qquad \forall e.
\end{equation}
As before, we add tau terms to the boundary conditions rather than enforcing them exactly, which allows us to absorb any incompatibilities in the specified RHS data.
Since the boundary conditions are now on faces, the tau terms added to each boundary condition take the form of the interior tau terms in the 2D problem.
That is, we modify the boundary conditions to:
\begin{equation}
    \beta_f u + \tau_f = g_f, \qquad \forall f \in \delta \Omega,
\end{equation}
where $\tau_f \in \polyq{N-2}{N}^{(2)}$.
These tau-modified boundary conditions add $6(N-2)^2$ additional constraints to $u$, leaving $12(N-2) + 8$ remaining degrees of freedom.

In the 2D case, four additional corner conditions were added to determine the boundary tau terms. 
In 3D, we must similarly add conditions on each edge to determine how the boundary conditions on adjacent faces are made consistent.
Naively, a symmetric scheme analogous to \cref{eq.isotropic_corners} could be imposed on each edge, yielding $12N$ additional constraints.
However, this system would then be overdetermined -- fully imposing constraints on the edges $(T,N)$ and $(N,E)$ would transitively imply additional constraints on the $(T,E)$ edge at the $TNE$ corner, which would be degenerate with the constraints on that edge.
Instead, the edge conditions should themselves each have tau terms $\tau_e \in \polyq{N-2}{N}$.
The tau-corrected symmetric edge conditions are therefore:
\begin{equation}
    2 \beta_f \beta_{f'} u + \tau_e = \beta_f g_{f'} + \beta_{f'} g_f \qquad \forall e.
\end{equation}
These tau-modified edge conditions then add $12(N-2)$ constraints on $u$, leaving 8 remaining degrees of freedom.

The final 8 constraints come from imposing conditions on the 8 vertices of the cube.
As with the previous conditions, these can be chosen in multiple fashions, but the octahedrally symmetric choice is to cross-apply the boundary operators from all three adjacent faces and sum, giving:
\begin{equation}
    3 \beta_f \beta_{f'} \beta_{f''} u = \beta_f \beta_{f'} g_{f''} + \beta_{f'} \beta_{f''} g_f + \beta_{f''} \beta_f g_{f'} \qquad \forall v=(f,f',f'').
\end{equation}

To summarize, the interior equations supply $(N-2)^3$ constraints, the boundary conditions on each face supply $6(N-2)^2$ constraints, the edge conditions supply $12(N-2)$ constraints, and the corner conditions supply 8 constraints.
In total, these combine to provide $N^3$ constraints that fully determine $u$.
This system can be applied to any choice of the tau polynomials that properly span the necessary quotient space at each level.
In particular, this method works for both collocation schemes using the boundary-supported Lagrange polynomials and Galerkin schemes like the ultraspherical method (see \cref{fig.constraints_3d}).

\section{Generalization to arbitrary dimensions and elliptic operators}
\label{sec.gen}

From the 2D and 3D cases, we can understand the extension of this system of imposing boundary conditions in any dimension and for any order elliptic operator.
We now consider an order-$b$ elliptic equation in $d$ dimensions on the hypercube $\Omega = [-1, 1]^d$:
\begin{equation}
    \mathcal{L} u(\vec{x}) = f(\vec{x}),
\end{equation}
where $\vec{x} = (x_1,...,x_d)$.
The equation is closed with $b$-many boundary conditions in each dimension:
\begin{equation}
    \beta_{j}^{(i)} u = g_{j}^{(i)}, \qquad i=1..b, \quad j=1..d.
\end{equation}
We consider cases where the boundary operators commute with those from all other dimensions:
\begin{equation}
    \beta_{j}^{(i)} \beta_{j'}^{(i')} = \beta_{j'}^{(i')} \beta_{j}^{(i)} \qquad i,i'=1 .. b, \quad j,j'=1..d, \quad j \neq j'.
\end{equation}

\paragraph{Trial discretization} 
We discretize $u$ using a direct-product trial basis as:
\begin{equation}
    u(\vec{x}) = \sum_{\vec{n} \in C_N^d} u_{\vec{n}} \phi_{\vec{n}}(\vec{x}),
\end{equation}
where $C_N^d = \{(n_1,...,n_d) : 0 \leq n_i < N\}$ and $\phi_{\vec{n}}(\vec{x}) = \prod_{i=1}^d \phi_{n_i}(x_i)$.
This discretization includes $N^d$ degrees of freedom.

\paragraph{Interior taus} 
We begin by adding a tau term to the PDE to allow the incorporation of the boundary conditions:
\begin{equation}
    \mathcal{L} u(\vec{x}) + \tau(\vec{x}) = f,
\end{equation}
where $\tau \in \polyq{N-b}{N}^{(d)}$.
The modified equation provides $(N-b)^d$ constraints on $u$.

\paragraph{Boundary taus} 
Next we add tau terms to the boundary conditions to allow the incorporation of the edge conditions in lower dimensions:
\begin{equation}
    \beta_j^{(i)} u + \tau_j^{(i)} = g_j^{(i)}, \qquad i=1 .. b, \quad j=1 .. d,
\end{equation}
where $\tau_j^{(i)} \in \polyq{N-b}{N}^{(d-1)}$.
Each modified boundary condition provides $(N-b)^{d-1}$ and there are $b d$ many boundary conditions, giving $b d (N-b)^{d-1}$ total constraints on $u$.

\paragraph{Interface conditions} 
To close the boundary conditions, we must impose conditions at the intersections of the faces in every dimension from $d-2$ down to 0.
At each dimension, we must impose conditions on every combination of boundary conditions and every combination of restricted dimensions.
These can be constructed in a hyperoctahedrally symmetric fashion by cross applying the boundary restriction operators to the unmodified boundary conditions and summing over all permutations for each choice of restricted dimensions at each level.
Tau terms are included all the way until dimension 0.
As in the case for 2D Poisson, this system is equivalent to imposing constraints on the tau terms from the previous levels which, along with the constraints from the commutator of the restrictions of the boundary conditions, produce a solvable system for the boundary tau terms.

For instance, at dimension $d-2$, we impose
\begin{equation}
    2 \beta_j^{(i)} \beta_{j'}^{(i')} u + \tau_{j,j'}^{(i,i')} = \beta_j^{(i)} g_{j'}^{(i')} + \beta_{j'}^{(i')} g_j^{(i)} \qquad i,i'=1 .. b, \quad j,j'=1 .. d, \quad j < j',
\end{equation}
where $\tau_{j,j'}^{(i,i')} \in \polyq{N-b}{N}^{(d-2)}$.
Each condition at this level provides $(N-b)^{d-2}$ constraints.
At this level, there are $b^2 \binom{d}{2}$ separate conditions that must be imposed, giving $b^2 \binom{d}{2} (N-b)^{d-2}$ total constraints on $u$.

At dimension $d-3$ we impose the conditions
\begin{equation}
    3 \beta_j^{(i)} \beta_{j'}^{(i')} \beta_{j''}^{(i'')} u + \tau_{j,j',j''}^{(i,i',i'')} = \beta_j^{(i)} \beta_{j'}^{(i')} g_{j''}^{(i'')} + \beta_{j'}^{(i')} \beta_{j''}^{(i'')} g_j^{(i)} + \beta_{j''}^{(i'')} \beta_j^{(i)} g_{j'}^{(i')}
\end{equation}
where $\tau_{j,j',j''}^{(i,i',i'')} \in \polyq{N-b}{N}^{d-3}$.
Each of these conditions provides $(N-b)^{d-3}$ constraints.
At this level, there are $b^3 \binom{d}{3}$ separate conditions that must be imposed, giving $b^3 \binom{d}{3} (N-b)^{d-3}$ total constraints on $u$.

At this point we can recognize the pattern: the interior equation, boundary conditions, and hierarchy of interface conditions provide constraints on $u$ which are enumerated by the binomial theorem.
Namely,
\begin{equation}
    N^d = \sum_{k=0}^d \binom{d}{k} b^k (N-b)^{d-k}.
\end{equation}
The $k=0$ term corresponds to the tau-modified interior PDE which provides $(N-b)^d$ constraints.
The $k=1$ term corresponds to the tau-modified boundary conditions which provide $b d (N-b)^{d-1}$ constraints.
The remaining terms correspond to the tau-modified interface conditions, all the way down to the $k=d$ term which provides $b^d$ conditions on the solution at the vertices.
This system is thus easily automatable once the boundary operators and all bases (test, trial, interior tau, boundary tau) are chosen.

\section{Examples}
\label{sec.examples}

We pick a variety of elliptic test problems in 2D and 3D to illustrate the flexibility of the symmetric formulations of the ultraspherical tau method.
All examples are computed with the Dedalus code \cite{Burns.2020} and the scripts are available on GitHub\footnote{\url{https://github.com/kburns/corner_taus}}.

\subsection{2D Poisson equation with various boundary conditions}

First, we consider the test problem posed in \cite{Fortunato.2019}, namely the 2D Poisson equation on $\Omega = [-1, 1]^2$:
\begin{equation}
    \Delta u(x,y) = f(x,y),
\end{equation}
with homogeneous Dirichlet boundary conditions and the given forcing
\begin{equation}
    f(x,y) = -100 x \sin(20 \pi x^2 y) \cos(4 \pi (x+y)).
\end{equation}

We discretize and solve this equation with the symmetric ultraspherical tau method as described above.
We use Chebyshev polynomials as our trial basis ($\phi_i = T_i$) and index-2 ultraspherical polynomials as our test basis ($\psi_i = C_i^{(2)}$).
We pick interior tau polynomials as the index-2 ultraspherical polynomials ($P_i = C_{N-i}^{(2)}$), the natural basis for the PDE in the ultraspherical method.
We pick the boundary tau polynomials as the Chebyshev polynomials ($Q_i = T_{N-i}$), the natural basis for the boundary conditions.

\begin{figure}
\centering
\includegraphics[width=0.99\linewidth]{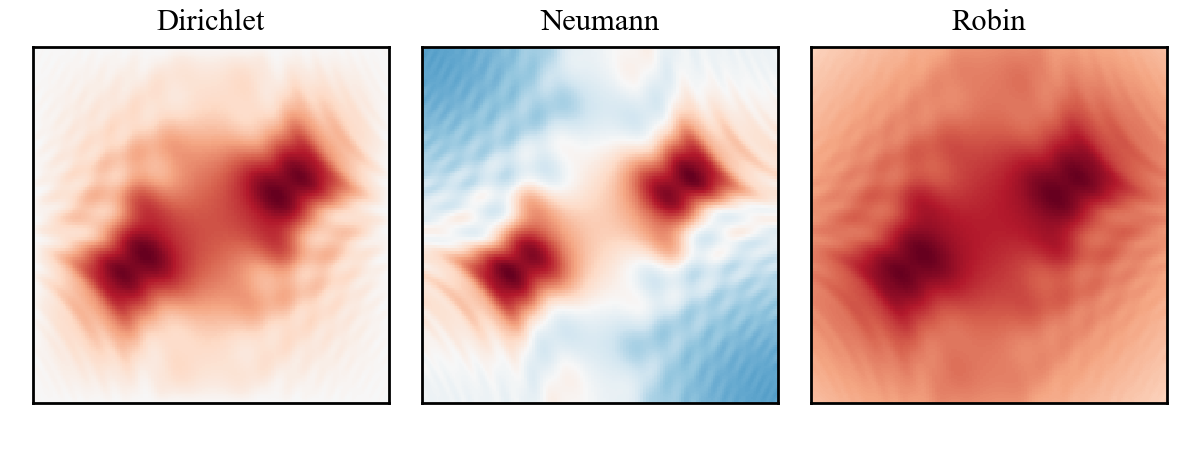}
\caption{Solutions to the forced 2D Poisson equation, computed using the symmetric ultraspherical tau method.
The solution is computed using homogeneous Dirichlet (top left), Neumann (top right), and Robin (bottom left) boundary conditions on each edge.}
\label{fig.example_2d_sols}
\end{figure}

To demonstrate the flexibility of this approach, we solve the PDE with the original homogeneous Dirichlet boundary conditions as well as homogeneous Neumann and Robin boundary conditions.
The solutions for each case are shown in \cref{fig.example_2d_sols}.
The infinity-norm self-convergence of the solutions to that with $N_\mathrm{ref} = 192$ are shown in \cref{fig.example_error}~(left).
In each case, we see supergeometric convergence over 10 decades from the Nyquist limit near $N=50$ until precision effects become important near $N=128$.
We emphasize that the only changes in the solver between the different cases is the definition of the boundary condition operators (and the fixing of an overall gauge $\int u \,\mathrm{d}\vec{x} = 0$ in the Neumann case) -- the tau structures and corner conditions remain identical.

\subsection{3D inhomogeneous biharmonic equation}

\begin{figure}
\centering
\includegraphics[width=0.49\linewidth]{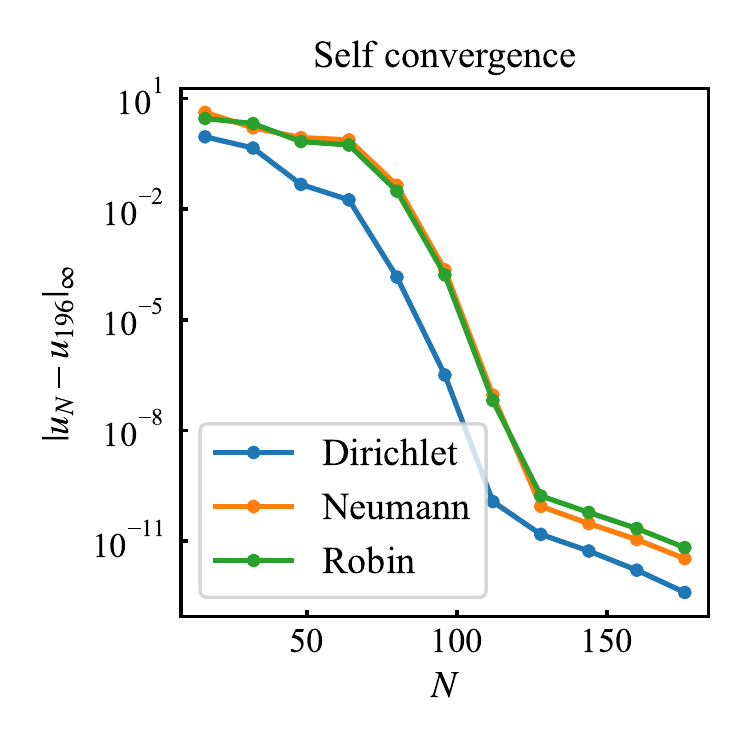}
\includegraphics[width=0.49\linewidth]{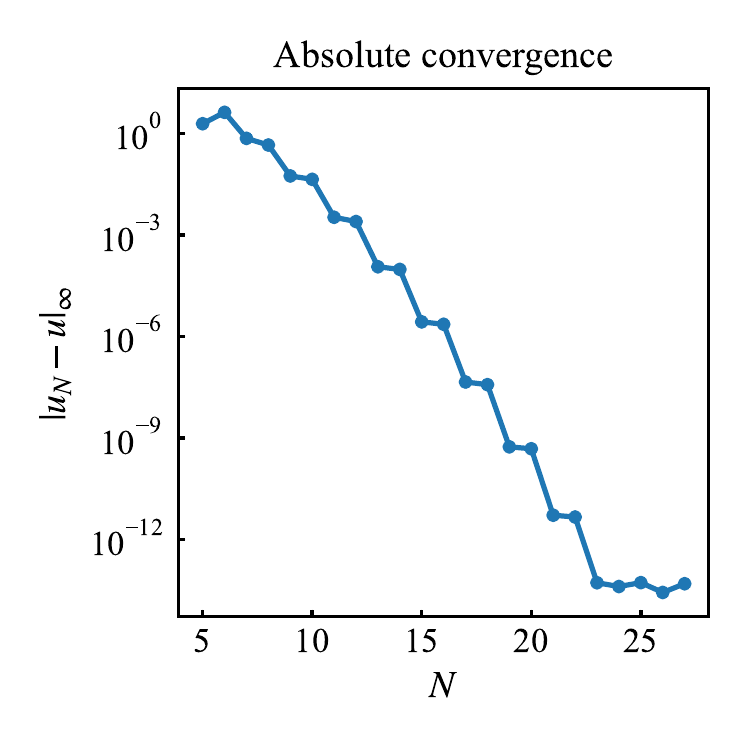}
\caption{Convergence rates for the symmetric ultraspherical tau method for the example problems.
Left: self-convergence for the forced 2D Poisson equation with various boundary conditions.
The error is computed relative to reference solutions with $N_\mathrm{ref}=192$.
Right: convergence for the 3D inhomogeneous biharmonic equation.
The error is computed relative to a simple manufactured solution.
All cases show supergeometric convergence.}
\label{fig.example_error}
\end{figure}

Next we solve the 3D inhomogenious biharmonic equation on $\Omega = [0,1]^3$:
\begin{equation}
    \Delta^2 u(x,y,z) = f(x,y,z).
\end{equation}
We pick a mixed set of boundary operators to illustrate the flexibility of the method: 
\begin{equation}
\begin{matrix}
    \beta_{x}^{(1)} = I_0, & \beta_{x}^{(2)} = I_1, & \beta_{x}^{(3)} = I_0 \partial_x^2, & \beta_{x}^{(4)} = I_1 \partial_x^2, \\
    \beta_{y}^{(1)} = I_0 \partial_{y}, & \beta_{y}^{(2)} = I_1 \partial_y, & \beta_{y}^{(3)} = I_0 \partial_y^3, & \beta_{x}^{(4)} = I_1 \partial_y^3, \\
    \beta_{z}^{(1)} = I_0, & \beta_{z}^{(2)} = I_1, & \beta_{z}^{(3)} = I_0 \partial_z, & \beta_{z}^{(4)} = I_1 \partial_z,
\end{matrix}
\end{equation}
A manufactured solution satisfying homogeneous boundary conditions with these boundary operators is:
\begin{equation}
    u = \sin(2 \pi x) \cos(2 \pi y) (1 - \cos(2 \pi z)),
\end{equation}
and the corresponding RHS term is $f = - (2 \pi)^4 \sin(2 \pi x) \cos(2 \pi y) (9 \cos(2 \pi z) - 4)$.

We discretize and solve the equation with this $f$ using the symmetric ultraspherical tau method.
We use Chebyshev polynomials as our trial basis ($\phi_i = T_i$) and index-4 ultraspherical polynomials as our test basis ($\psi_i = C_i^{(4)}$).
We pick interior tau polynomials as the index-4 ultraspherical polynomials ($P_i = C_{N-i}^{(4)}$), the natural basis for the PDE in the ultraspherical method.
We pick the boundary and edge tau polynomials as the Chebyshev polynomials ($Q_i = T_{N-i}$), the natural basis for the boundary conditions and their cross-combinations.

The convergence of the discrete solution to the true solution is shown in \cref{fig.example_error}~(right).
We see that the method converges supergeometrically over many decades until precision effects become important near $N = 25$.
We again emphasize that the boundary operators can be modified at will without needing to change the tau terms or interface (edge and corner) conditions.

\subsection{2D biharmonic spectrum with various tau bases}

As discussed previously, the generalized tau method enables separating the choice of tau polynomials, which formally determines the solution of the modified PDE over polynomials, from the test and trial functions used to represent and solve for this polynomial solution.
Here we demonstrate this capability by solving for the eigenvalues $\sigma$ of the 2D biharmonic equation on $\Omega = [0,1]^2$:
\begin{equation}
    \Delta^2 u(x,y) = \sigma \, u(x,y).
\end{equation}
We pick simply supported boundary conditions on all sides:
\begin{equation}
\begin{matrix}
    \beta_{x}^{(1)} = I_0, & \beta_{x}^{(2)} = I_1, & \beta_{x}^{(3)} = I_0 \partial_x^2, & \beta_{x}^{(4)} = I_1 \partial_x^2, \\
    \beta_{y}^{(1)} = I_0, & \beta_{y}^{(2)} = I_1, & \beta_{y}^{(3)} = I_0 \partial_y^2, & \beta_{x}^{(4)} = I_1 \partial_y^2.
\end{matrix}
\end{equation}
The eigenfunctions of this system take the form $u(x,y) = \sin(m \pi x) \sin(n \pi y)$ with corresponding eigenvalues $\sigma = \pi^4 (m^2 + n^2)^2$ for $m,n \in \mathbb{N}$.

We discretize and solve the equation using the symmetric ultraspherical tau method.
We use Chebyshev polynomials as our trial basis ($\phi_i = T_i$) and index-4 ultraspherical polynomials as our test basis ($\psi_i = C_i^{(4)}$).
We pick the interior tau polynomials as ultraspherical polynomials ($P_i = C_{N-i}^{(\alpha)}$) with $\alpha$ varying from 0 to 4.
We pick the boundary tau polynomials as the Chebyshev polynomials ($Q_i = T_{N-i}$), the natural basis for the boundary conditions.

\begin{figure}
\centering
\includegraphics[width=0.7\linewidth]{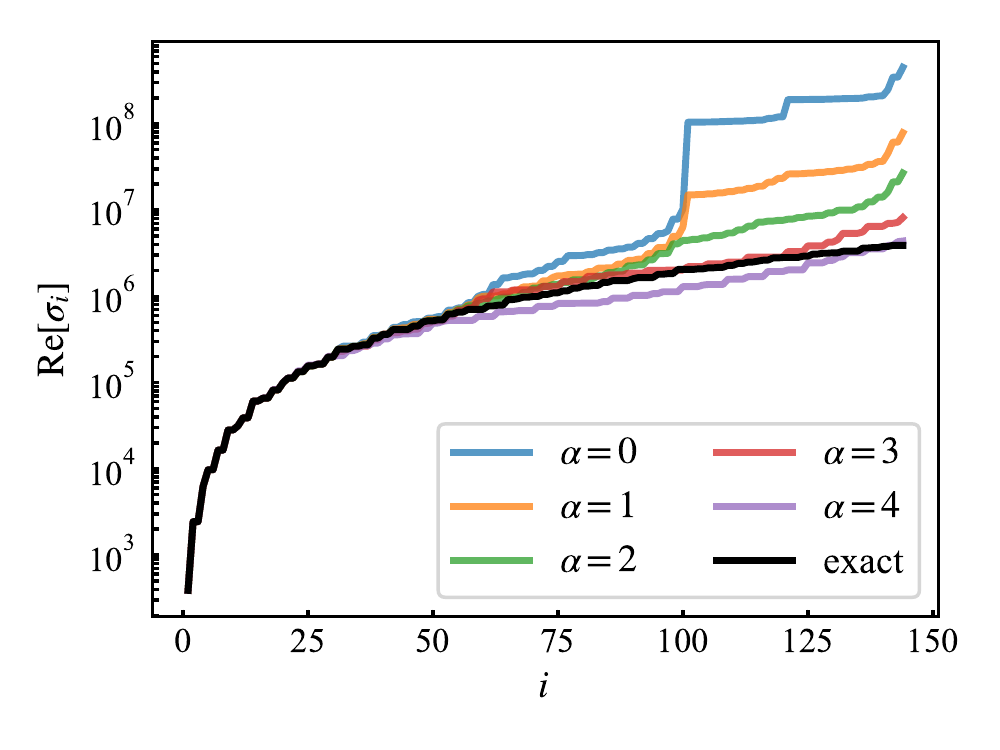}
\caption{Computed and exact eigenvalues for the 2D biharmonic equation with simply supported boundary conditions.
The symmetric ultraspherical tau method is used with different interior tau polynomials from the ultraspherical families, as indicated by $\alpha$.}
\label{fig.biharm_evals}
\end{figure}

\cref{fig.biharm_evals} shows the resulting discrete eigenvalues as a function of $\alpha$.
We see that the choice of tau polynomials effects the accuracy of the spectrum for large $m$ and $n$.
In particular, low $\alpha$ result in highly spurious modes.
Rectangular collocation and the classical tau method correspond to $\alpha = 0$, which has a highly compromised spectrum.
The ultraspherical tau method with $\alpha = 4$ has less error in the eigenvalues, but increasing $\alpha$ does lower the accuracy in the intermediate ``resolved'' portion of the spectrum.

The optimal choice of tau polynomial for most equations remains an open question.
The generalized tau method presented here enables the study and comparison of different tau choices in a common numerical framework.
By studying the impact of different tau choices for common equations, we aim to eventually enable the fully automatic inclusion of boundary conditions in an optimal fashion for general polynomial spectral methods.

\section*{Acknowledgments}
We thank Ben Brown, Daniel Lecoanet, Jeff Oishi, Sheehan Olver, and Alex Townsend for helpful discussions.

\bibliographystyle{siamplain}
\bibliography{references}

\begin{thebibliography}{10}

\bibitem{Awan.1993}
{\sc M.~A. Awan and T.~N. Phillips}, {\em {Well-conditioned spectral
  discretizations of the biharmonic operator}}, International Journal of
  Computer Mathematics, 48 (1993), pp.~179--189,
  \url{https://doi.org/10.1080/00207169308804201}.

\bibitem{Boyd.2001}
{\sc J.~P. Boyd}, {\em {Chebyshev and Fourier Spectral Methods}}, Dover
  Publications, Dover Publications, 00 2001,
  \url{http://adsabs.harvard.edu/abs/2001cfsm.book.....B}.

\bibitem{Burns.2020}
{\sc K.~J. Burns, G.~M. Vasil, J.~S. Oishi, D.~Lecoanet, and B.~P. Brown}, {\em
  {Dedalus: A flexible framework for numerical simulations with spectral
  methods}}, Physical Review Research, 2 (2020), p.~838,
  \url{https://doi.org/10.1103/physrevresearch.2.023068},
  \url{https://link.aps.org/doi/10.1103/PhysRevResearch.2.023068}.
\newblock 40 pages, 18 figures.

\bibitem{Dang-Vu.1993}
{\sc H.~Dang-Vu and C.~Delcarte}, {\em {An Accurate Solution of the Poisson
  Equation by the Chebyshev Collocation Method}}, Journal of Computational
  Physics, 104 (1993), pp.~211--220,
  \url{https://doi.org/10.1006/jcph.1993.1021}.

\bibitem{Driscoll.2015}
{\sc T.~A. Driscoll and N.~Hale}, {\em {Rectangular spectral collocation}}, IMA
  Journal of Numerical Analysis, 38 (2015), p.~dru062,
  \url{https://doi.org/10.1093/imanum/dru062},
  \url{http://imajna.oxfordjournals.org/content/early/2015/02/06/imanum.dru062.full}.

\bibitem{Fortunato.2021}
{\sc D.~Fortunato, N.~Hale, and A.~Townsend}, {\em {The ultraspherical spectral
  element method}}, Journal of Computational Physics, 436 (2021), p.~110087,
  \url{https://doi.org/10.1016/j.jcp.2020.110087},
  \url{https://arxiv.org/abs/2006.08756}.

\bibitem{Fortunato.2019}
{\sc D.~Fortunato and A.~Townsend}, {\em {Fast Poisson solvers for spectral
  methods}}, IMA Journal of Numerical Analysis, 40 (2017), pp.~1994--2018,
  \url{https://doi.org/10.1093/imanum/drz034},
  \url{http://arxiv.org/abs/1710.11259v1},
  \url{https://arxiv.org/abs/1710.11259}.

\bibitem{GIBSON.2009}
{\sc J.~F. GIBSON, J.~HALCROW, and P.~CVITANOVIĆ}, {\em {Equilibrium and
  travelling-wave solutions of plane Couette flow}}, Journal of Fluid
  Mechanics, 638 (2009), pp.~243--266,
  \url{https://doi.org/10.1017/s0022112009990863},
  \url{https://arxiv.org/abs/0808.3375}.

\bibitem{Gillman.2015}
{\sc A.~Gillman, A.~H. Barnett, and P.-G. Martinsson}, {\em {A spectrally
  accurate direct solution technique for frequency-domain scattering problems
  with variable media}}, BIT Numerical Mathematics, 55 (2015), pp.~141--170,
  \url{https://doi.org/10.1007/s10543-014-0499-8}.

\bibitem{Haidvogel.1979xt}
{\sc D.~B. Haidvogel and T.~Zang}, {\em {The accurate solution of poisson's
  equation by expansion in chebyshev polynomials}}, Journal of Computational
  Physics, 30 (1979), pp.~167--180,
  \url{https://doi.org/10.1016/0021-9991(79)90097-4}.

\bibitem{Haldenwang.1984}
{\sc P.~Haldenwang, G.~Labrosse, S.~Abboudi, and M.~Deville}, {\em {Chebyshev
  3-D spectral and 2-D pseudospectral solvers for the Helmholtz equation}},
  Journal of Computational Physics, 55 (1984), pp.~115--128,
  \url{https://doi.org/10.1016/0021-9991(84)90018-4}.

\bibitem{Julien.1996}
{\sc K.~Julien, S.~Legg, J.~Mcwilliams, and J.~Werne}, {\em {Rapidly rotating
  turbulent Rayleigh-Bénard convection}}, Journal of Fluid Mechanics, 322
  (1996), pp.~243--273, \url{https://doi.org/10.1017/s0022112096002789}.

\bibitem{Julien.2009}
{\sc K.~Julien and M.~Watson}, {\em {Efficient multi-dimensional solution of
  PDEs using Chebyshev spectral methods}}, Journal of Computational Physics,
  228 (2009), pp.~1480 -- 1503,
  \url{https://doi.org/10.1016/j.jcp.2008.10.043},
  \url{http://adsabs.harvard.edu/cgi-bin/nph-data\_query?bibcode=2009JCoPh.228.1480J\&link\_type=EJOURNAL}.

\bibitem{Moin.1980}
{\sc P.~Moin and J.~Kim}, {\em {On the numerical solution of time-dependent
  viscous incompressible fluid flows involving solid boundaries}}, Journal of
  Computational Physics, 35 (1980), pp.~381 -- 392,
  \url{https://doi.org/10.1016/0021-9991(80)90076-5},
  \url{http://adsabs.harvard.edu/cgi-bin/nph-data\_query?bibcode=1980JCoPh..35..381M\&link\_type=EJOURNAL}.

\bibitem{Olver.2013}
{\sc S.~Olver and A.~Townsend}, {\em {A Fast and Well-Conditioned Spectral
  Method}}, SIAM Review, 55 (2013), pp.~462 -- 489,
  \url{https://doi.org/10.1137/120865458},
  \url{http://epubs.siam.org/doi/abs/10.1137/120865458}.

\bibitem{Orszag.19714ni}
{\sc S.~A. Orszag}, {\em {Accurate solution of the Orr–Sommerfeld stability
  equation}}, Journal of Fluid Mechanics, 50 (1971), pp.~689--703,
  \url{https://doi.org/10.1017/s0022112071002842}.

\bibitem{Orszag.1980hv}
{\sc S.~A. Orszag and L.~C. Kells}, {\em {Transition to turbulence in plane
  Poiseuille and plane Couette flow}}, Journal of Fluid Mechanics, 96 (1980),
  pp.~159--205, \url{https://doi.org/10.1017/s0022112080002066}.

\bibitem{Townsend.2015yvf}
{\sc A.~Townsend and S.~Olver}, {\em {The automatic solution of partial
  differential equations using a global spectral method}}, Journal of
  Computational Physics, 299 (2015), pp.~106 -- 123,
  \url{https://doi.org/10.1016/j.jcp.2015.06.031},
  \url{http://adsabs.harvard.edu/cgi-bin/nph-data\_query?bibcode=2015JCoPh.299..106T\&link\_type=EJOURNAL}.

\bibitem{trefethen2000spectral}
{\sc L.~N. Trefethen}, {\em Spectral methods in MATLAB}, SIAM, 2000.

\bibitem{Trefethen.1993}
{\sc L.~N. Trefethen, A.~E. Trefethen, S.~C. Reddy, and T.~A. Driscoll}, {\em
  {Hydrodynamic Stability Without Eigenvalues}}, Science, 261 (1993),
  pp.~578--584, \url{https://doi.org/10.1126/science.261.5121.578}.

\bibitem{Tuckerman.1989}
{\sc L.~S. Tuckerman}, {\em {Divergence-free velocity fields in nonperiodic
  geometries}}, Journal of Computational Physics, 80 (1989), pp.~403--441,
  \url{https://doi.org/10.1016/0021-9991(89)90108-3}.

\end{thebibliography}

\end{document}